\documentclass{amsart}

\usepackage{amsmath,amssymb,amsthm,amsfonts,mathrsfs}
\usepackage{relsize}
\usepackage{enumerate}
\usepackage{fullpage}
\usepackage[colorlinks=true,
linkcolor=blue,
anchorcolor=blue,
citecolor=red
]{hyperref}
\allowdisplaybreaks 
\newtheorem{theorem}{Theorem}[section]
\newtheorem{lemma}[theorem]{Lemma}
\newtheorem{corollary}[theorem]{Corollary}
\newtheorem{proposition}[theorem]{Proposition}

\theoremstyle{definition}

\newtheorem{example}[theorem]{Example}
\theoremstyle{remark}
\newtheorem{remark}[theorem]{Remark}
\newcommand{\N}{\mathbb{N}}
\newcommand{\Z}{\mathbb{Z}}
\newcommand{\Q}{\mathbb{Q}}
\newcommand{\R}{\mathbb{R}}
\newcommand{\C}{\mathbb{C}}
\newcommand{\F}{\mathbb{F}}
\renewcommand{\P}{\mathbb{P}}
\newcommand{\cG}{\mathcal{G}}
\newcommand{\cP}{\mathcal{P}}
\newcommand{\cM}{\mathcal{M}}
\newcommand{\pl}{\partial}
\newcommand{\AVG}{\frac{1}{N} \sum_{n=1}^{N}}
\newcommand{\BEu}[1]{\underset{#1}{\mathlarger{\mathlarger{\mathbb{E}}}^{~}}\,}
\newcommand{\BEul}[1]{\underset{#1}{\mathlarger{\mathlarger{\mathbb{E}}}^{\text{\normalfont\footnotesize log}}}\,}
\newcommand{\ve}{\varepsilon}
\newcommand{\on}{\operatorname}

\newcommand{\lcm}{\on{lcm}}
\renewcommand{\mod}[1]{\,(\on{mod}#1)}
\newcommand{\of}[1]{\left(#1\right)}
\newcommand{\set}[1]{\left\{#1\right\}}
\newcommand{\abs}[1]{\left\vert#1\right\vert}

\author[B. Wang]{Biao Wang}
\address{Hua Loo-Keng Center for Mathematical Sciences, Academy of Mathematics and Systems Science, Chinese Academy of Sciences, Beijing 100190, China}
\email{wangbiao@amss.ac.cn}
\thanks{Project funded by China Postdoctoral Science Foundation under grant number 2021TQ0350.}
\date{\today}

\makeatletter
\@namedef{subjclassname@2020}{\textup{}2020 Mathematics Subject Classification}
\makeatother

\title[The number of prime divisors]{Dynamics on the number of prime divisors for additive arithmetic semigroups}

\begin{document}
	
\begin{abstract}
In 2020, Bergelson and Richter gave a dynamical generalization of the classical Prime Number Theorem, which has been generalized by Loyd in a disjoint form with the Erd\H{o}s-Kac Theorem. These generalizations reveal the rich ergodic properties of  the number of prime divisors of integers.	In this article, we show a new generalization of Bergelson and Richter's Theorem in a disjoint form with the distribution of the largest prime factors of integers.  Then following Bergelson and Richter's techniques, we will show the analogues of all of these results  for the arithmetic semigroups arising from finite fields as well.
\end{abstract}
	
\maketitle

\section{Introduction and statement of results}
Let $n\in \N:=\set{1,2,3,\dots}$ be a natural number. Let $\Omega(n)$ be the total number of prime factors of $n$ counted with multiplicities. The distribution of values of $\Omega(n)$ involves some important information on the distribution of prime numbers. It is well-known (e.g., \cite{Landau1953, Mangoldt1897}) that the Prime Number Theorem (PNT) is equivalent to the assertion that
\begin{equation}\label{landau}
	\lim_{N\to\infty}\frac1N\sum_{n=1}^N\lambda(n)=0,
\end{equation}
where $\lambda(n):=(-1)^{\Omega(n)}$ is the Liouville function. That is, there are as many natural number $n\in \N$ for which $\Omega(n)$ is even  as $n\in \N$ for which $\Omega(n)$ is odd. 
In 2020, Bergelson and Richter \cite{BergelsonRichter2020} gave us a beautiful dynamical generalization of this form of the Prime Number Theorem as follows.   
\begin{theorem}[{\cite[Theorem A]{BergelsonRichter2020}}]
	\label{thm_BR2020thmA}
	Let $(X,\mu,T)$ be a uniquely ergodic additive topological dynamical system. Then we have
	\begin{equation}\label{eqn_BR2020thmA}
		\lim_{N\to\infty}\frac1N\sum_{n=1}^Nf(T^{\Omega(n)}x)=\int_Xf\,d\mu
	\end{equation}
	for every $x\in X$ and $f\in C(X)$. Here  we denote by $C(X)$ the space of continuous functions on $X$.
\end{theorem}

Theorem~\ref{thm_BR2020thmA} tells us that the sequence $\set{T^{\Omega(n)}x}_{n\in\N}$ is uniformly
distributed in any uniquely ergodic additive topological dynamical system $X$ for every point $x\in X$. If one takes $(X,T)$ to be the rotation on two points, then \eqref{landau} is recovered from \eqref{eqn_BR2020thmA}. Theorem~\ref{thm_BR2020thmA} also unifies many classical results in multiplicative number theory including a theorem of Pillai and Selberg and a theorem of Erd\H{o}s and Delange, see Sect.~\ref{sec_BR_applications}. Recently, Loyd \cite{Loyd2021} has generalized Bergelson and Richter's Theorem in a disjoint form of Theorem~\ref{thm_BR2020thmA} with the Erd\H{o}s-Kac Theorem as follows. Let $C_c(\R)$ denote the set of compactly supported continuous functions on $\R$. 

\begin{theorem}[{\cite[Theorem 1.3]{Loyd2021}}]
	\label{thm_Loyd2021}
	Let $(X, \mu, T)$ be uniquely ergodic, and let $F \in C_c(\R)$. Then we have
	\begin{equation}\label{eqn_Loyd2021}
		\lim_{N \to \infty} \frac1N\sum_{n=1}^N F \Big( \frac{\Omega(n) - \log \log N }{\sqrt{\log \log N}} \Big) f(T^{\Omega(n) }x) 
		=
		\Big(\frac{1}{\sqrt{2\pi}} \int_{-\infty}^{\infty} F(t) e^{-t^2/2} \, dt\Big)\Big( \int_X f \, d\mu \Big)
	\end{equation}
	for all $f \in C(X)$ and $x \in X$. 
\end{theorem}
If one takes $f$ to be a nonzero constant function, then one gets the Erd\H{o}s-Kac Theorem. Two sequences $a,b: \N \to \C$ are called \textit{asymptotically independent} if 
\[
\lim_{N\to\infty}\left(\AVG a(n) \overline{b(n)} - \bigg( \AVG a(n) \bigg)\bigg( \AVG \overline{b(n)} \bigg) \right)=0.
\]
By Theorem~\ref{thm_Loyd2021}, the sequences $\set{F \Big( \frac{\Omega(n) - \log \log N }{\sqrt{\log \log N}} \Big)}_{n=1}^N$ and $\set{f(T^{\Omega(n) }x)}_{n\in \N}$ are asymptotically independent. In this paper, we will  show a new disjoint form of Theorem~\ref{thm_BR2020thmA}  with the distribution of the largest prime factors of integers.

The largest prime factors of integers are well-distributed. Let $n\ge1$ be an integer. Denote by $p_{\max}(n)$ the largest prime factor of $n$ for $n\ge2$, and set $p_{\max}(1)=1$. In 1977, Alladi \cite[Theorem~1]{Alladi1977} showed that $p_{\max}(n)$ is equidistributed in arithmetic progressions. In 2020, Kural, McDonald and Sah \cite[Theorem~3.1]{KuralMcDonaldSah2020} generalized Alladi's result to the natural density over number fields. Let $S$ be a set of primes.  We say $S$ has \textit{natural density} $\delta(S)$ if the limit $\delta(S):=\lim_{N\to\infty}\#\set{p\le N: p\in S}/\pi(N)$ exists, where $\pi(N)$ denotes the number of primes up to $N$. Over the rational field $\Q$, Kural et al.'s result states that if $S$ is a set of primes of natural density $\delta(S)$, then
\begin{equation}
	\label{eqn_KMS}
	\lim_{N\to\infty}\frac1N\sum_{\substack{1\le n\le N\\ p_{\max}(n)\in S}} 1= \delta(S).
\end{equation}
This means that the largest prime factors of integers are well distributed along natural numbers. Moreover, \eqref{eqn_KMS} is an important step to establish Alladi-type formulas, which are another kind of generalizations of the Prime Number Theorem,
see \cite{Alladi1977, KuralMcDonaldSah2020, Wang2021jnt}. Recently, in an unpublished note with  Pan Yan, Liyang Yang and Shaoyun Yi, following the approaches in \cite{Alladi1977, KuralMcDonaldSah2020}, we have discovered that the following asymptotic estimate
\begin{equation}
	\label{eqn_Huristic}
	\sum_{\substack{1\le n\le N\\ p_{\max}(n)\in S}}f(n)\sim\delta(S)\sum_{1\le n\le N }f(n)
\end{equation}
holds for any divisor-bounded  multiplicative function $f(n)$ satisfying $f(p)=\alpha$ for all primes $p$, where $\alpha>0$ is some constant. In this article, for the function $f(T^{\Omega(n)}x)$, though it is not multiplicative, we observe that $f(T^{\Omega(p)}x)=f(Tx)$ is a constant for any prime $p$. At the same time, Bergelson and Richter's Theorem tells us the orbits $\set{T^{\Omega(n)}x}_{n\in\N}$ are uniformly distributed  for any $x\in X$. These facts lead us to speculate that the asymptotic estimate \eqref{eqn_Huristic} holds as well for $f(T^{\Omega(n)}x)$ and $x\in X$. 

\begin{theorem}
	\label{thm_Wang2021}
	Let $(X,\mu,T)$ be uniquely ergodic. Let $F:\N\to\C$ be a bounded arithmetic function and $\delta$ a constant such that 
	$$\sum_{1\le n\le N}F(p_{\max}(n))\sim \delta \cdot N.$$
	Then we have
	\begin{equation}\label{eqn_Wang2021}
		\lim_{N\to\infty}\frac1N\sum_{n=1}^NF(p_{\max}(n))f(T^{\Omega(n)}x)=\delta\int_X f\,d\mu
	\end{equation}
	for every $x\in X$ and $f\in C(X)$. In particular, if $S$ is a set of primes of natural density $\delta(S)$, then we have
	\begin{equation}\label{eqn_Wang2021_S}
		\lim_{N\to\infty}\frac1N\sum_{\substack{1\le n\le N\\p_{\max}(n)\in S}}f(T^{\Omega(n)}x)=\delta(S)\int_X f\,d\mu
	\end{equation}
	for every $x\in X$ and $f\in C(X)$.
\end{theorem}

By Theorem~\ref{thm_Wang2021}, the sequences $\set{F(p_{\max}(n))}_{n\in\N}$ and $\set{f(T^{\Omega(n) }x)}_{n\in \N}$ are asymptotically independent. In general, one may have the asymptotic independence between $\set{F(p_{\max}(n))}_{n\in\N}$ and $\set{g(S_ny)}_{n\in \N}$ for a finitely generated and strongly uniquely ergodic  multiplicative system $(Y,S)$ defined by Bergelson and Richter \cite{BergelsonRichter2020} and $y\in Y$.

Moreover, in number theory, it has long been recognized that algebraic function fields are natural analogues of ordinary algebraic number fields in many aspects.  The Prime Number Theorem is such a typical instance. Since  Theorems~\ref{thm_BR2020thmA}, \ref{thm_Loyd2021} and \ref{thm_Wang2021} are generalizations of the PNT, we are interested in the analogues of these results for finite fields and function fields. In this article, we will establish the analogues of \eqref{eqn_BR2020thmA}, \eqref{eqn_Loyd2021}, \eqref{eqn_Wang2021} and \eqref{eqn_Wang2021_S}
for additive arithmetic semigroups arising from finite fields. We state the results as follows.

An \textit{additive arithmetical semigroup} (cf. \cite{KnopfmacherZhang2001}) is a free Abelian semigroup $\cG$ with identity element $1$, generated  by a countable set $\cP$ of \textit{primes} in $\cG$ such that every non-identity element in $\cG$ uniquely factorizes into a finite product of powers of elements of $\cP$, together with a \textit{degree} mapping $\pl: \cG\to\N\cup\set{0}$ on $\cG$ such that
\begin{enumerate}[(1)]
	\item $\pl(1)=0, \pl(p)>0$ for $p\in\cP$;
	\item $\pl(gh)=\pl(g)+\pl(h)$ for all $g,h\in\cG$;
	\item For each $x>0$, the set $\set{g\in\cG:\pl(g)\le x}$ is finite.
\end{enumerate}

Let $\cG$ be an additive arithmetical semigroup. Let $\cG(n)$ be the set of elements in $\cG$ of degree $n$ and let $G(n)=|\cG(n)|$ be the number of elements in $\cG$ of degree $n$. We say $\cG$ is of \textit{Axiom $\mathcal{A}^\#$} type if there exist constants $A>0, q>1,0\le\eta<1$ depending only on $\cG$ such that 
\begin{equation}
	\label{eqn_axiom}
	G(n)=A q^n+O(q^{\eta n}).
\end{equation}
The set of monic polynomials over a finite field is a classical prototype of additive arithmetic semigroups of Axiom $\mathcal{A}^\#$ type, see Sect.~\ref{sec_aas_examples} for more examples.  Let $\cP(n)$ be the set of primes in $\cG$ of degree $n$ and let $\pi_\cG(n)=|\cP(n)|$ be the number of primes in $\cG$ of degree $n$. Let $\bar{\Lambda}(n)$ be the analogue of the Chebyshev function for $\cG$ defined by $$\bar{\Lambda}(n):=\sum_{\substack{p\in\cP, r\ge1\\ \pl(p^r)=n}}\pl(p).$$ By Lemma 3.5.2, Theorem 3.4.5 and Theorem 5.1.1 in \cite{KnopfmacherZhang2001}, we have the following two types of abstract Prime Number Theorems under the assumption \eqref{eqn_axiom}:
\begin{align}
	\bar{\Lambda}(n)&=q^n+O(q^{\theta n});\label{eqn_classicalPNT}\\
	\bar{\Lambda}(n)&=q^n\big(1+(-1)^{n+1}\big)+O(q^{\theta n}) \label{eqn_nonclassicalPNT}
\end{align}
for some $\theta$ with $\eta<\theta<1$. We call \eqref{eqn_classicalPNT} a classical type of PNT, from which we have $\pi_\cG(n)\sim q^n/n$ as $n\to\infty$. Eq. \eqref{eqn_nonclassicalPNT} is said to be non-classical. We only consider the PNT of classical type in this paper.

For any $g$ in $\cG$, let $\Omega(g)$ be the total number of prime factors of $g$ with multiplicities counted. Let $\cG_n$ be the set of elements in $\cG$ of degree at most $n$, then $\cG_n$ is finite. The reason why we consider $\cG_n$ instead of $\cG(n)$ in the following theorems is that we employ primes of \textit{all} degrees in our proofs. The following theorem gives us the analogues of Bergelson and Richter's Theorem  (Theorem~\ref{thm_BR2020thmA}) and Loyd's Theorem (Theorem~\ref{thm_Loyd2021}) for additive arithmetic semigroups.  

\begin{theorem}
	\label{thm_Loyd2021_ff}
	Let $\cG$ be an additive arithmetic semigroup of Axiom $\mathcal{A}^\#$ type admitting a classical type of Prime Number Theorem.	Let $(X,\mu,T)$ be a uniquely ergodic additive topological dynamical system. Let $F\in C_c(\R)$, then we have
	\begin{align}
		\lim_{n\to\infty}\frac1{|\cG_n|}\sum_{ g\in\cG_n}f(T^{\Omega(g)}x)&=\int_Xf\,d\mu \label{eqn_BR2020_ff}\\
		\lim_{n\to\infty}\frac1{|\cG_n|}\sum_{ g\in\cG_n}F\Big(\frac{\Omega(g)-\log n}{\sqrt{\log n}}\Big)f(T^{\Omega(g)}x)&=\Big(\frac1{\sqrt{2\pi}}\int_{-\infty}^\infty F(t)e^{-t^2/2}\,dt\Big)\Big(\int_Xf\,d\mu\Big)\label{eqn_Loyd2021_ff}
	\end{align}
	for every $x\in X$ and $f\in C(X)$.
\end{theorem}

By \eqref{eqn_BR2020_ff}, the orbit $\set{T^{\Omega(g)}x}_{g\in \cG}$ is uniformly distributed in the uniquely ergodic additive topological dynamical system $X$ for every point $x\in X$.  If one takes $f$ to be a constant function in \eqref{eqn_Loyd2021_ff}, then one gets back to the analogue of the  Erd\H{o}s-Kac Theorem (e.g., see \cite{Liu2004}) for $\cG$. Taking $\cG$ to be the set of monic polynomials over a finite field, as a corollary of  Theorem~\ref{thm_Loyd2021_ff},  we get the following analogue of Bergelson and Richter's Theorem  for finite fields.

\begin{corollary}
	Let $\F_q$ be a finite field of $q$ elements. For any monic polynomial $M$, let $\Omega(M)$ be the total number of  irreducible monic polynomial factors of $M$ with multiplicities counted. Let $\cM_n$ be the set of monic polynomials in $\F_q[x]$ of degree at most $n$. Let $(X,\mu,T)$ be uniquely ergodic.  Then  we have
	\begin{equation}
		\lim_{n\to\infty}\frac1{|\cM_n|}\sum_{M\in \cM_n}f(T^{\Omega(M)}x)=\int_Xf\,d\mu
	\end{equation}
	for every $x\in X$ and $f\in C(X)$.
\end{corollary}

Let $\cG$ be as in Theorem~\ref{thm_Loyd2021_ff}. Taking $(X,T)$ to be the rotation by $m$ points, i.e. $X=\set{0,1,\dots,m-1}$ and $T(x)=x+1\mod{m}$, and taking $f=1_{r\, {\rm{mod}}\, m},x=0$ in \eqref{eqn_BR2020_ff} for $0\le r\le m-1$, one gets the following density theorem.  It is a generalization of \cite[Theorem 4.1.1]{KnopfmacherZhang2001} dealing with the case $m=2$. For a subset $H$ of $\cG$, we say $H$ has \textit{natural density} $\delta(H)$ if the limit  $\delta(H):=\lim_{n\to\infty}|H\cap \cG_n|/|\cG_n|$ exists.

\begin{corollary} 
	\label{cor_equidistribution_modm_ff}
	Given an integer $m\ge2$. Then for any $0\le r\le m-1$, the set of elements in $\cG$ such that $\Omega(g)\equiv r\mod{m}$ has natural density $1/m$ in $\cG$.
\end{corollary}

A sequence $\set{a(g)}_{g\in\cG}\subset \R$ is said to be \textit{uniformly distributed mod 1} over $\cG$ if
\begin{equation}
	\label{eqn_def_ud_mod1_ff}
	\lim_{n\to\infty}\frac1{|\cG_n|}\sum_{ g\in\cG_n}f(a(g))=\int_0^1f(t)dt
\end{equation}
for all continuous functions $f:\R/\Z\to\C$. Let  $\alpha$ be an irrational number. Taking $(X,T)$ to be the rotation by $\alpha$ in Theorem~\ref{thm_Loyd2021_ff} and using Weyl's criterion \cite{Weyl1916}, we get the uniform distribution mod $1$ of $\set{\Omega(g)\alpha}_{g\in\cG}$.

\begin{corollary} 
	Let  $\alpha\in\R\setminus\Q$. Then $\set{\Omega(g)\alpha}_{g\in\cG}$ is uniform distributed mod 1 over $\cG$.
\end{corollary}

Now, we introduce the analogue of Theorem~\ref{thm_Wang2021} for additive arithmetic semigroups. Let $\cG$ be an additive arithmetic semigroup. Let $\cP$ be the set of primes in $\cG$. For an element $g\in \cG$, let $\pl_{\max}(g):=\max\set{\pl(p):p|g,p\in\cP}$ be the largest degree of the prime factors of $g$, and let $Q(g)$ be the total number of prime factors of $g$ attaining the largest degree $\pl_{\max}(g)$. For example, if $g=p_1^{\alpha_1}p_2^{\alpha_2}p_3^{\alpha_3}$ with $\pl(p_1)<\pl(p_2)=\pl(p_3), p_1,p_2,p_3\in\cP$, then $Q(g)=\alpha_2+\alpha_3$. Let $\cG^+:=\set{g\in \cG: Q(g)=1}$. For $g\in \cG^+$, we denote by $p_{\max}(g)$ the prime factor of $g$ attaining the largest degree, and we call $p_{\max}(g)$ the largest prime factor of $g$. For natural numbers, it is a fact that the power of $p_{\max}(n)$ in the prime factorization of $n$ is $1$ for almost all $n\in \N$, see  \cite[Theorem (1.7)]{IvicPomerance1984}. Analogously, in Sect.~\ref{sec_pmax} we will show that almost all elements of $\cG$ are in $\cG^+$. In other words, the complement $\cG\setminus \cG^+$ is of zero natural density. For a subset $S$ of primes in $\cP$, we say $S$ has \textit{natural density}\footnote{If the limit in \eqref{eqn_def_density} exists, then the following limit $\lim\limits_{n\to\infty}\frac{|S\cap \cG_n|}{|\cP\cap \cG_n|}$ exits and equals to $\delta(s)$. But the converse is not true.A counterexample is the set of all primes of even degrees.} $\delta(S)$ if the following limit exists
\begin{equation}\label{eqn_def_density}
	\delta(S):=\lim_{n\to\infty}\frac{|S\cap \cG(n)|}{\pi_\cG(n)}.
\end{equation} 

Let $\cG^+_n$ be the set of elements in $\cG^+$ of degree at most $n$. In Sect.~\ref{sec_pmax} we will show that if $S$ is a set of primes in $\cP$ of natural density $\delta(S)$, then the following analogue of Eq. \eqref{eqn_KMS} holds
\begin{equation}\label{eqn_equidistribution_density}
	\lim_{n\to\infty}\frac1{|\cG_n^+|}\sum_{\substack{g\in \cG^+_n\\ p_{\max}(g)\in S}} 1= \delta(S).
\end{equation}

Now, we state the analogue of Theorem~\ref{thm_Wang2021} for $\cG$ as follows. 

\begin{theorem}
	\label{thm_Wang2021_ff}
	Let $\cG$ be an additive arithmetic semigroup of Axiom $\mathcal{A}^\#$ type admitting a classical type of Prime Number Theorem. Let $F:\cG\to\C$ be a bounded function and $\delta$ a constant  such that 
	$$\sum_{g\in \cG_n^+}F(p_{\max}(g))\sim \delta \cdot |\cG_n^+|.$$
	Let $(X,\mu,T)$ be uniquely ergodic. Then we have
	\begin{equation}\label{eqn_Wang2021_pmax_ff}
		\lim_{n\to\infty}\frac1{|\cG_n^+|}\sum_{g\in \cG_n^+}F(p_{\max}(g))f(T^{\Omega(g)}x)=\delta\int_X f\,d\mu
	\end{equation}
	for every $x\in X$ and $f\in C(X)$.	In particular, if $S$ is a set of primes in $\cP$ of natural density $\delta(S)$, then we have
	\begin{equation}\label{thm_Wang2021_ff_eq1}
		\lim_{n\to\infty}\frac1{|\cG_n^+|}\sum_{\substack{ g\in\cG_n^+\\p_{\max}(g)\in S}}f(T^{\Omega(g)}x)=\delta(S)\int_Xf\,d\mu.
	\end{equation}
\end{theorem}

\begin{remark}
	One may replace $\Omega(g)$ by $\omega(g)$ for Theorems~\ref{thm_Loyd2021_ff} and  \ref{thm_Wang2021_ff}, where $\omega(g)$ is the number of distinct prime factors of $g$. For the other analogues of these results, we leave the interested readers to investigate the arithmetic semigroups of Axiom $\mathcal{A}$ type (cf. \cite{Knopfmacher1975}) or other types that admit the Prime Number Theorem.
\end{remark}

\noindent\textbf{Structure of the paper.} In Sect.~\ref{sec_preliminaries}, we will introduce some background materials on unique ergodicity, applications of Theorem~\ref{thm_BR2020thmA}, and  additive arithmetic groups. Then in Sect.~\ref{sec_pmax}, using an elementary estimate for smooth elements we will show that $\cG\setminus \cG^+$ is of zero natural density and that Eq. \eqref{eqn_equidistribution_density} holds. 
In Sect.~\ref{sec_BR_techniques_pf_Wang2021}, we cite Bergelson and Richter's techniques in \cite[Theorem A]{BergelsonRichter2020} and apply them to show Theorem~\ref{thm_Wang2021}. Then we establish these techniques for additive arithmetic semigroups in Sect.~\ref{sec_BR_techniques_ff} and prove Theorems~\ref{thm_Loyd2021_ff} and \ref{thm_Wang2021_ff} in Sect.~\ref{sec_proof_mainthm_ff}. In their proof of Theorem~\ref{thm_BR2020thmA}, Bergelson and Richter use the $\rho$-adic intervals $[\rho^j,\rho^{j+1})$ to break the line $[1,\infty)$; in our proof of Theorems~\ref{thm_Loyd2021_ff} and \ref{thm_Wang2021_ff}, we use the set $\cG(j)$ of elements of degree $j$ to break $\cG$ up into stages.

\noindent\textbf{Notation.} In this paper,  the symbol $1_{\text{statement } P}$ is equal to 1 if the statement $P$ is true and zero otherwise.  The cardinality of a set $S$ is denoted by $|S|$ or $\# S$. We denote by $a_n\sim b_n \,\,(n\to\infty)$ if $\lim_{n\to\infty}a_n/b_n=1$. The big-$O$ notation $f(n) = O(g(n)$ or $f(n)\ll g(n)$ means that there is a positive constant $C>0$ such that $|f(n)| \le C|g(n)|$ for all $n\ge1$, and
$C$ is called the \textit{implied constant} of this $O$-term. The implied constant $C$ may depend on $\cG$ or other parameters, but it does not depend on $n$. The little-$o$ notation $f(n)=o(g(n))$ or $f(n)=o_{n\to\infty}(g(n))$ means that $\lim_{n\to\infty}f(n)/g(n)=0$.

\section{Preliminaries}
\label{sec_preliminaries}

\subsection{Unique Ergodicity}

Let $X$ be a compact metric space, and let $T:X\to X$ be a continuous map. Then the pair $(X,T)$ is called an \textit{additive topological dynamical system}. A Borel probability measure $\mu$ on $X$ is called \textit{T-invariant} if $\mu(T^{-1}A) = \mu(A)$ for all measurable subsets $A\subset X$. Every additive topological dynamical system has at least one $T$-invariant measure due to the Bogolyubov-Krylov theorem (e.g., \cite[Corollary 6.9.1]{Walters1982}). If a topological system $(X, T)$ admits only one $T$-invariant measure $\mu$, then $(X, T)$ is called \textit{uniquely ergodic}. It is well-known (e.g., \cite[Theorem 6.19]{Walters1982}) that $(X, \mu, T)$ is uniquely ergodic if and only if 
\begin{equation}\label{eqn_ergodicity}
	\lim_{N \to \infty} \frac1N\sum_{n=1}^N  f(T^n x) = \int_X f \, d\mu
\end{equation}
holds for  all $x \in X$ and $f\in C(X)$. Bergelson and Richter's Theorem tells us that Eq. \eqref{eqn_ergodicity} also holds if the orbit $\set{T^nx: n\in\N}$ is replaced by the orbit $\set{T^{\Omega(n)}x: n\in\N}$ along $\Omega(n)$ for every point $x\in X$.

\subsection{Applications of Theorem~\ref{thm_BR2020thmA}}
\label{sec_BR_applications}
Let $m\ge2$. Taking $(X,T)$ to be the rotation on $m$ points in Theorem~\ref{thm_BR2020thmA}, one obtains the following theorem due to Pillai and Selberg.

\begin{theorem}[Pillai \cite{Pillai1940}, Selberg \cite{Selberg1939}]
	\label{thm_PS}
	Let $m\ge2$ be an integer. Then for any $0\le r\le m-1$, the set of natural numbers  $n\in\N$ such that $\Omega(n)\equiv r\mod{m}$ has natural density $\frac1m$.	
\end{theorem}

A sequence $\{a(n)\}_{n \in \N} \subset \R$ is \textit{uniformly distributed mod 1} if 
\[
\lim_{N \to \infty} \frac1N\sum_{n=1}^N f(a(n)) = \int_0^1 f(t)dt 
\]
for all continuous functions $f: \R/\Z \to \C$. 
Taking $(X,T)$ to be the rotation by $\alpha$  in Theorem~\ref{thm_BR2020thmA}, one obtains the following theorem which was first mentioned by Erd\H{o}s \cite[p. 2]{Erdos1946} without proof
but later proven by Delange \cite{Delange1958}.
\begin{theorem}[Erd\H{o}s \cite{Erdos1946}, Delange \cite{Delange1958}]
	\label{thm_ED}
	Let $\alpha \in \R\setminus \Q$. Then $\{\Omega(n)\alpha\}_{n \in \N}$ is uniformly distributed mod 1. 
\end{theorem}

In \cite{BergelsonRichter2020}, Bergelson and Richter give a polynomial generalization of Theorem~\ref{thm_ED} by combining Furstenberg's method with Theorem~\ref{thm_BR2020thmA}. It is a variant of Weyl's theorem \cite{Weyl1916} asserting that  a polynomial sequence $Q(n)=c_k n^k + \ldots + c_1 n + c_0$, $n\in\N$, is uniformly distributed mod~$1$ if and only if at least one of the coefficients $c_1,\ldots,c_k$ is irrational.

\begin{theorem}[{\cite[Corollary 1.2]{BergelsonRichter2020}}]
	\label{thm_polynomoial_ud_Omega}
	Let $Q(n)=c_k n^k + \ldots + c_1 n + c_0$ with $c_0,\dots, c_k\in \R$. Then the sequence $\set{Q(\Omega(n))}_{n\in\N}$ is uniformly distributed mod~$1$ if and only if at least one of the coefficients $c_1,\ldots,c_k$ is irrational.
\end{theorem}

We refer readers to Bergelson and Richter’s article \cite{BergelsonRichter2020} for more applications of Theorem~\ref{thm_BR2020thmA}.

\subsection{Additive arithmetic semigroups}
\label{sec_aas_examples}

Let $\cG$ be an additive arithmetical semigroup with degree mapping $\pl$. We note that by definition $\pl(g)=0$ if and only if $g=1$ for any $g\in\cG$. Fix a positive $q>1$. Define $|g|:=q^{\pl(g)}$ for all $g\in \cG$, then it endows a \textit{norm} $|\cdot|$ on $\cG$. For $g,h\in\cG$, write $g=\prod_{i} p_i^{\alpha_i}$ and $h=\prod_{i}p_i^{\beta_i}$, where $\alpha_i,\beta_i\ge0$ for all $i$. We denote by $g|h$, if $\alpha_i\le\beta_i$ for all $i$. If $g|h$, then we call $g$ a \textit{factor} or \textit{divisor} of $h$ and $h$ a \textit{multiple} of $g$. We say $d:=\gcd(g,h)$ is the \textit{greatest common divisor} of $g$ and $h$, if for any common divisor $d'$ of $g$ and $h$ we have $d'|d$.  We say $m:=\lcm(g,h)$ is the \textit{least common multiple} of $g$ and $h$, if for any common multiple $m'$ of $g$ and $h$ we have $m|m'$.  One can verify that $\gcd(g,h)=\prod_{i}p_i^{\min\{\alpha_i,\beta_i\}}$, $\lcm(g,h)=\prod_{i}p_i^{\max\{\alpha_i,\beta_i\}}$, and $gh=\gcd(g,h)\lcm(g,h)$. Therefore, we have $|g||h|/|\lcm(g,h)|=|\gcd(g,h)|$ for all $g,h\in\cG$.

Next, we list some examples of additive arithmetical semigroups of Axiom $\mathcal{A}^\#$ type with classical Prime Number Theorems.  Details of illustrations on the last two examples can be found in \cite{DuanMaYi2021}. And we refer readers to Knopfmacher and Zhang's book \cite{KnopfmacherZhang2001} for more examples. Let $\F_q$ be a finite field of $q$ elements with $q$ a prime power.

\begin{example}[Finite fields]
	Let $\cG$ be the set of all monic polynomials in $\F_q[x]$. Then it is an additive arithmetical semigroup with 
	$G(n)=q^n$. And the Prime Polynomial Theorem asserts that 
	$$\pi_\cG(n)=\frac{q^n}n+O\of{\frac{q^{n/2}}n}.$$	
\end{example}

\begin{example}[Function fields]
	Let $K$ be a global
	function field with constant field $\F_q$. A prime divisor or prime in $K$ is a discrete valuation ring $R_P$ containing $\F_q$ with maximal ideal $P$ such that the quotient field of $R_P$ is $K$. The size of the residue field $R_P/P$ is a power $q^{\pl(P)}$ of $q$ for some exponent $\pl(P)$. Let $\cG$  be the free Abelian semigroup generated by prime divisors of $K$. That is, $\cG$ is the set of the so-called effective divisors on $K$. Then $(\cG,\pl)$ is an additive arithmetical semigroup. By the Riemann-Roch theorem, it is well-known (cf. Rosen \cite{Rosen2002}, say) that 
	$$G(n)=h_K\frac{q^{n-g_K+1}-1}{q-1}$$
	for $n>2g_K-2$, where $g_K$ is the genus of $K$, and $h_K$ is the class number of $K$. Hence $G(n)=c_Kq^n+O(1)$,
	where $c_K=h_Kq^{1-g_K}/(q-1)$, which means $\cG$ is of Axiom $\mathcal{A}^\#$ type. And the Prime Number Theorem for function fields asserts that
	$$\pi_\cG(n)=\frac{q^n}n+O\of{\frac{q^{n/2}}n}.$$
\end{example}

\begin{example}[Algebraic varieties]
	Let $X$ be a $d$-dimensional projective smooth variety defined over $\F_q$, $d\ge1$. A \textit{prime 0-cycle} $P$ of $X$ is the finite sum of all the distinct Galois conjugations of a geometric point in $X(\overline{\F}_q)$ over the ground field $\F_q$. The \textit{degree} $\pl(P)$ is the minimal
	positive integer $r$ such that $P$ splits over $\F_{q^r}$. Let $\cG$ be the free Abelian semigroup generated by prime 0-cycles of $X$. Then applying the Weil conjecture \cite{Weil1949} settled by a series of works of Dwork, Grothendieck and Deligne \cite{Dwork1960, Grothendieck1995, Deligne1974}, one can deduce that
	$$\pi_\cG(n)=\frac{q^{dn}}n+O\of{\frac{q^{(d-\frac12)n}}n}.$$
	By \cite{Indlekofer1991}, we have
	$$G(n)=c_Xq^{dn}+O\of{q^{(d-\frac12)n}}$$
	for some constant $c_X>0$. Hence $\cG$ is an additive arithmetic semigroup of Axiom $\mathcal{A}^\#$ type.
\end{example}

\begin{example}[Grassmannian varieties]
	Let $\ell>k\ge1$ be integers.  For any linear vector space $V$ over $\F_q$, let $G(k + 1, V )$ be the set of $(k+1)$-dimensional linear subspaces of $V$. For every $k+1$-dimensional subspace $\tilde{H}$ of $\mathbb{A}_{\overline{\F}_q}^{\ell+1}$, let $H$ be the finite union of all the distinct Galois conjugates of $\tilde{H}$. Then $H$ is  a reduced variety over $\F_q$. The degree of $H$ is defined to be the number of geometric irreducible components of $H$. 
	Let $\cP_{k,\ell}:=\set{H: \tilde{H}\in G(k+1,\mathbb{A}_{\overline{\F}_q}^{\ell+1})}$, and let $\cG_{k,\ell}$ be the set of finite scheme unions of varieties in  $\cP_{k,\ell}$. Then $\cG_{k,\ell}$ is an additive arithmetical semigroup of Axiom $\mathcal{A}^\#$ type (see \cite{DuanMaYi2021}) with 
	$$G(n)=c_{\cG_{k,\ell}}q^{dn}+O\of{q^{\eta dn}}$$
	for some constant $c_{\cG_{k,\ell}}>0$ and $0<\eta<1$, where $d=(k+1)(\ell-k)$. And 
	$$\pi_\cG(n)=\frac{q^{dn}}n+O\of{\frac{q^{\eta dn}}n}.$$
	A remark is that by the Pl\"ucker embedding \cite[Lecture 6]{Harris1995}, one can identity  $\cG_{k,\ell}$ to the semigroup of 0-cycles of the Grassmannian variety $\mathbb{G}(k,\P_{\F_q^\ell})$.
\end{example}

\section{The prime factors of largest degree}
\label{sec_pmax}

Let $\cG$ be an additive arithmetic semigroup satisfying Eq. \eqref{eqn_axiom}. For $g\in \cG$, if $\pl_{\max}(g)\le m$ then we call $g$ an \textit{$m$-smooth element} in $\cG$.  Let $\Psi_\cG(n,m):=\#\set{g\in \cG: \pl(g)\le n,\pl_{\max}(g)\le m}$ be the number\footnote{The definition here is slightly different from the customary $\Psi_\cG(n,m)$ defined by $\#\set{g\in \cG: \pl(g)= n,\pl_{\max}(g)\le m}$.} of $m$-smooth elements in $\cG_n$. By {\cite[Theorem 3.3.3]{KnopfmacherZhang2001}}, we have the following
Mertens-type estimate for $\cG$:
\begin{equation}\label{mertens_eq}
	\sum_{\pl(p)\le n}\frac{\pl(p)}{|p|}=n+O(1).
\end{equation}
Then using Rankin's method (cf. \cite[Theorem III.5.1]{Tenenbaum2015} or \cite[Theorem 4.6]{TenenbaumMendes-France2000}), we get the following upper bound for $\Psi_\cG(n,m)$.

\begin{lemma} 
	\label{lem_smooth_numbers}
	For $n\ge m>2/\log q$, we have
	\begin{equation}\label{eqn_smooth_numbers}
		\Psi_\cG(n,m)\ll q^ne^{-\frac{n}{2m}}.
	\end{equation}
\end{lemma}
\begin{proof}  Let $\chi(g,m)$ be the multiplicative function on $\cG$ defined by
	$$\chi(g,m):=\begin{cases}
		1, & \text{ if } \pl_{\max}(g)\le m;\\
		0, & \text{ if } \pl_{\max}(g)> m.
	\end{cases}$$
	Then $\Psi_\cG(n,m)=\sum_{g\in\cG_n}\chi(g,m)$. For all $\alpha\ge0$, we have
	\begin{align}
		\Psi_\cG(n,m)&\le \sum_{\pl(g)\le\frac34n}1+\sum_{\pl(g)\le n}\left(\frac{|g|}{q^{\frac34n}}\right)^\alpha\chi(g,m)\nonumber\\
		&\ll q^{\frac34n}+ q^{-\frac34n\alpha}\sum_{\pl(g)\le n}|g|^\alpha\chi(g,m).\label{sm_pf_eq1}
	\end{align}
	
	Take $\alpha:=2/(3m\log q)$, then $q^{-\frac34n\alpha}=e^{-\frac{n}{2m}}$.
	Let $f_\alpha(g)=\chi(g,m)|g|^\alpha\prod_{p|g}(1-|p|^{-\alpha})$. It is a multiplicative function on $\cG$. Then one can readily verify that 
	\begin{equation}
		|g|^\alpha\chi(g,m)\le \sum_{h|g}f_\alpha(h).
	\end{equation}
	It follows that
	\begin{align}
		\sum_{\pl(g)\le n}|g|^\alpha\chi(g,m)&\le 	\sum_{\pl(g)\le n}\sum_{h|g}f_\alpha(h)\nonumber\\
		&=\sum_{\pl(h)\le n}f_\alpha(h)|\cG_{n-\pl(h)}|\nonumber\\
		&\ll q^n\sum_{\pl(h)\le n}\frac{f_\alpha(h)}{|h|}
		\nonumber\\
		&\le q^n\sum_{\pl_{\max}(h)\le m}\frac{f_\alpha(h)}{|h|}
		\nonumber\\
		&=q^n\prod_{\pl(p)\le m}\Big(1+\sum_{j=1}^\infty \frac{f_\alpha(p^j)}{|p|^j}\Big). \label{sm_pf_eq2}
	\end{align}
	Notice that 
	$$\frac{f_\alpha(p^j)}{|p|^j}=\frac{|p|^{j\alpha}(1-|p|^{-\alpha})}{|p|^j}\le \frac{|p|^\alpha \alpha \log|p|}{|p|}(|p|^{\alpha-1})^{j-1}, j\ge1.$$
	By the choice of $\alpha$, we have $|p|^\alpha\le e^{2/3}$ and $|p|^{\alpha-1}<|p|^{-1/3}\le q^{-1/3}$. This gives us that $$\sum_{j=1}^\infty \frac{f_\alpha(p^j)}{|p|^j}=O\Big(\frac{\alpha\pl(p)}{|p|}\Big).$$
	Then by Mertens-type estimate \eqref{mertens_eq} for $\cG$, we have
	\begin{equation}\label{sm_pf_eq3}
		\prod_{\pl(p)\le m}\Big(1+\sum_{j=1}^\infty \frac{f_\alpha(p^j)}{|p|^j}\Big) \ll \exp\Big(O\Big(\sum_{\pl(p)\le m}\frac{\alpha\pl(p)}{|p|}\Big)\Big)\ll1.
	\end{equation}
	When $m>2/\log q$, we have $q^ne^{-n/(2m)}>q^ne^{-n\log q/4}=q^{3n/4}$.
	Thus, Eq. \eqref{eqn_smooth_numbers} follows by \eqref{sm_pf_eq1}, \eqref{sm_pf_eq2} and \eqref{sm_pf_eq3}.
\end{proof}

Next, we use Lemma~\ref{lem_smooth_numbers} to show that $\cG\setminus\cG^+$ is of zero natural density, which follows by the following lemma. 
\begin{lemma}
	\label{lem_size_of_cG}
	For any positive $0<\sigma<1$, we have
	\begin{equation}\label{pmaxerror_eq}
		|\cG_n|=|\cG^+_n|+O(q^nn^{-\sigma}).
	\end{equation}
\end{lemma}

\begin{proof}  We notice that $$|\cG_n|-|\cG^+_n|=|\cG_n\setminus \cG^+_n|,$$ 
	and we can write $|\cG_n\setminus \cG^+_n|$ in terms of $\Psi_\cG(n,m)$ as follows:
	\begin{align}
		|\cG_n\setminus \cG^+_n|&=\sum_{\substack{\pl(g)\le n\\Q(g)\ge2}}1=\sum_{\substack{p,q\in\cP,g\in\cG\\ \pl(pqg)\le n\\ \pl_{\max}(g)\le\pl(p)=\pl(q)}}1\nonumber\\
		&=\sum_{k\le \frac{n}2}\sum_{\substack{p,q\in\cP\\ \pl(p)=\pl(q)=k}}\sum_{\substack{g\in\cG\\ \pl(g)\le n-2k\\ \pl_{\max}(g)\le k}}1\nonumber\\
		&=\sum_{k\le \frac{n}2}\sum_{\substack{p,q\in\cP\\ \pl(p)=\pl(q)=k}} \Psi_\cG(n-2k,k)\nonumber\\
		&=\sum_{k\le \frac{n}2}\pi_\cG(k)^2 \Psi_\cG(n-2k,k). \label{pmaxerror_pf_eq1}
	\end{align}
	
	By \cite[Corollary 3.2.2]{KnopfmacherZhang2001}, we have  the Chebyshev-type upper bound $\pi_\cG(k)\ll q^k/k$ for $\pi_\cG(k)$. As regards $\Psi_\cG(n-2k,k)$, we will use the following three types of estimates for $\Psi_\cG(n,m)$. First, if $m\le C$ for some constant $C\ge 1$, then  $\Psi_\cG(n,m)\ll n^C$. Second, by Lemma~\ref{lem_smooth_numbers} we have $\Psi_\cG(n,m)\ll q^ne^{-\frac{n}{2m}}$ for $n\ge m>2/\log q$. Third, we have a trivial bound $\Psi_\cG(n,m)\le|\cG_n|\ll q^n$  holds uniformly for all $m$. For any positive $\sigma\in(0,1)$, we break Eq.~\eqref{pmaxerror_pf_eq1} up into three parts:
	\begin{align}
		|\cG_n\setminus \cG^+_n|&=\sum_{k\le \frac2{\log q}}\pi_\cG(k)^2 \Psi_\cG(n-2k,k)+ \sum_{\frac2{\log q}<k\le n^\sigma}\pi_\cG(k)^2 \Psi_\cG(n-2k,k)\nonumber\\
		&\qquad+\sum_{n^\sigma<k\le \frac{n}2}\pi_\cG(k)^2 \Psi_\cG(n-2k,k)\nonumber\\
		&\ll n^{2/\log q}+ \sum_{k\le n^\sigma}\frac{q^{2k}}{k^2} \cdot q^{n-2k} e^{-\frac{n-2k}{2k}}+ \sum_{k>n^\sigma}\frac{q^{2k}}{k^2} \cdot q^{n-2k}\nonumber\\
		&\ll n^{2/\log q}+ q^n\sum_{k\le n^\sigma}\frac1{k^2}e^{-\frac{n}{2k}}+ q^n\sum_{k>n^\sigma}\frac1{k^2}\nonumber\\
		&\ll n^{2/\log q}+ q^n e^{-\frac{n^{1-\sigma}}{2}}+ q^nn^{-\sigma}\nonumber\\
		&\ll q^nn^{-\sigma}.
	\end{align}
	Thus, we get \eqref{pmaxerror_eq}.
\end{proof}

Now, we use the ideas in \cite{KuralMcDonaldSah2020} and \cite{DuanWangYi2021} to show the following theorem which is equivalent to Eq.~\eqref{eqn_equidistribution_density}.

\begin{theorem}
	\label{thm_natural_density}
	Let $S$ be a set of primes of natural density $\delta(S)$. Then we have
	\begin{equation}\label{eqn_natural_density}
		\sum_{\substack{g\in \cG^+_n\\p_{\max}(g)\in S}}1=\delta(S) |\cG_n^+|+o(q^n).
	\end{equation}
\end{theorem}
\begin{proof} 
	First, we break the summation on the left-hand side of \eqref{eqn_natural_density} up into two parts:
	\begin{equation}\label{eqn_Sigma1_Sigma2}
		\sum_{\substack{g\in \cG^+_n\\p_{\max}(g)\in S}}1=\sum_{\substack{g\in \cG^+_n\\p_{\max}(g)\in S,\, \pl(p_{\max}(g))\le m}}1+\sum_{\substack{g\in \cG^+_n\\p_{\max}(g)\in S,\, \pl(p_{\max}(g))> m}}1:=\Sigma_1+\Sigma_2,
	\end{equation}
	where $1\le m\le n$ is to be determined.
	
	For $\Sigma_1$, by Lemma~\ref{lem_smooth_numbers}, we have
	\begin{equation}\label{eqn_Sigma_1}
		\Sigma_1\le \sum_{\substack{g\in \cG_n^+\\ \pl(p_{\max}(g))\le m}}1\le \Psi_\cG(n,m)\ll q^ne^{-\frac{n}{2m}}
	\end{equation}
	uniformly for any $S$.
	
	As regards $\Sigma_2$,	we write it  in terms of $\Psi_\cG(n,m)$ as follows:
	\begin{align}
		\Sigma_2&=\sum_{\substack{m<\pl(p)\le n\\ p\in S}}\sum_{\substack{g\in \cG^+_n\\p_{\max}(g)=p}}1\nonumber\\
		&=\sum_{\substack{m<\pl(p)\le n\\ p\in S}}\sum_{\substack{\pl(g)\le n-\pl(p)\\ \pl_{\max}(g)<\pl(p)}}1\nonumber\\
		&=\sum_{\substack{m<\pl(p)\le n\\ p\in S}}\Psi_\cG\big(n-\pl(p),\pl(p)-1\big)\nonumber\\
		&=\sum_{m<k\le n}\Psi_\cG(n-k,k-1)\pi_{\cG,S}(k) \label{eqn_Sigma_2},
	\end{align}
	where $\pi_{\cG,S}(k):=\#\set{p\in S:\pl(p)=k}$. 
	
	By \eqref{eqn_Sigma1_Sigma2}-\eqref{eqn_Sigma_2},
	we get that
	\begin{equation}\label{eqn_pf_Sigma}
		\sum_{\substack{g\in \cG^+_n\\p_{\max}(g)\in S}}1=\sum_{m<k\le n}\Psi_\cG(n-k,k-1)\pi_{\cG,S}(k)+O(q^ne^{-\frac{n}{2m}}).
	\end{equation}
	Taking $S=\cP$ in \eqref{eqn_pf_Sigma}, we have
	\begin{equation}\label{eqn_pf_cGplus}
		|\cG_n^+|=\sum_{m<k\le n}\Psi_\cG(n-k,k-1)\pi_\cG(k)+O(q^ne^{-\frac{n}{2m}}).
	\end{equation}
	
	It follows by \eqref{eqn_pf_Sigma} and \eqref{eqn_pf_cGplus} that
	\begin{equation}\label{eqn_Sigma3}
		\sum_{\substack{g\in \cG^+_n\\p_{\max}(g)\in S}}1 \,-\,\delta(S)|\cG_n^+|=\Sigma_3 +O(q^ne^{-\frac{n}{2m}}),
	\end{equation}
	where 
	$$\Sigma_3:=\sum_{m<k\le n}\Psi_\cG(n-k,k-1)(\pi_{\cG,S}(k)-\delta(S)\pi_\cG(k)).$$
	Then by the proof of \cite[Theorem 4.4]{DuanWangYi2021}, we have
	\begin{equation}\label{eqn_pf_Sigma3}
		\Sigma_3\ll nq^nv_{\cG,S}(m),
	\end{equation}
	where
	\begin{align*}
		e_{\cG,S}(n)&:=\sup_{k\le n}|\pi_{\cG,S}(k)-\delta(S)\pi_\cG(k)|,\\
		v_{\cG,S}(n)&:=\sup_{k\ge n}	\frac{e_{\cG,S}(k)}{k\pi_\cG(k)}
	\end{align*}
	for all $n\ge1$. From \eqref{eqn_Sigma3} and \eqref{eqn_pf_Sigma3}, we obtain that
	\begin{equation}
		\sum_{\substack{g\in \cG^+_n\\p_{\max}(g)\in S}}1 =\delta(S)|\cG_n^+|+O\of{ q^n\big(nv_{\cG,S}(m)+ e^{-\frac{n}{2m}}\big)}.
	\end{equation}
	
	As the argument in \cite[Theorem 4.4]{DuanWangYi2021} shows, there exists $m=k_n$ such that $\lim_{n\to\infty}nv_{\cG,S}(k_n)=0$ and $\lim_{n\to\infty}n/k_n=\infty$. Hence $O\of{ q^n\big(nv_{\cG,S}(m)+ e^{-\frac{n}{2m}}\big)}=o(q^n)$,
	and \eqref{eqn_natural_density} follows.
\end{proof}

\section{Bergelson and Richter's techniques and proof of Theorem~\ref{thm_Wang2021}}
\label{sec_BR_techniques_pf_Wang2021}

In this section, we first cite Bergelson and Richter's techniques which are developed to establish Theorem~\ref{thm_BR2020thmA}.  Then we apply them to show Theorem~\ref{thm_Wang2021} in Sect.~\ref{sec_pf_Wang2021}. In Sect.~\ref{sec_BR_techniques_ff}, we will show the analogue of these techniques for additive arithmetic groups.

\subsection{Bergelson and Richter's techniques}
\label{sec_BR_techniques}

For any positive $s>0$, we write $[s]:=\N\cap[1,s]$. Then the set $\{1,\ldots,N\}$ is denoted by $[N]$.
For a finite and non-empty set $B\subset \N$ and a function $a\colon B\to\C$, the \textit{Ces\`aro average} of $a$ over $B$ and the \textit{logarithmic average} of $a$ over $B$ are defined respectively by
$$
\BEu{n\in B} a(n) := \frac{1}{|B|}\sum_{n\in B} a(n)\qquad\text{and }\qquad
\BEul{n\in B} a(n) := \frac{\sum_{n\in B} {a(n)}/{n}}{\sum_{n\in B}{1}/{n}}.
$$

Denote by $\P$ the set of all primes in $\N$. Inspired by the following observation 
\begin{equation}
	\label{eqn_TK}
	\lim_{s\to\infty}\,\lim_{N\to\infty}\,\BEu{n\in [N]} \left|\BEul{p\in \P\cap [s]} \big(1- p 1_{p\mid n}\big)\right| \,=\,0
\end{equation}
due to Daboussi \cite[Lemma 1]{Daboussi1975} and K\'atai \cite[Eq. (3.1)]{Katai1986}, which follows from the Tur\'an-Kubilius inequality (cf. \cite[Lemma 4.1]{Elliott1971}), Bergelson and Richter generalize \eqref{eqn_TK} to the following inequality. It plays an important role in their proof of Theorem~\ref{thm_BR2020thmA}.

\begin{proposition}[{\cite[Proposition 2.1]{BergelsonRichter2020}}]
	\label{prop_BR_inequality}
	Let $B\subset\N$ be a finite and non-empty set of integers. Then for any bounded arithmetic function $a\colon\N\to\C$ with $|a|\le1$ we have that
	\begin{equation}
		\label{eqn_BR_inequality}
		\limsup_{N\to\infty}\,\left|\BEu{n\in [N]} a(n)\,-\,  \BEul{m\in B}\, \BEu{n\in [\frac{N}{m}]} a(mn) \right| \,\le\,\left( \BEul{m\in B}\BEul{n\in B}\Phi(n,m) \right)^{1/2},
	\end{equation}
	where $\Phi(m,n):=\gcd(m,n)-1$.
\end{proposition}

For the error term $\BEul{m\in B}\BEul{n\in B} \Phi(m,n)$, the higher chance of being coprime to each other for  integers in $B$,  the smaller the error is. The following Lemma provides a precise criterion on how to estimate it.

\begin{lemma}[{\cite[Lemmas 2.5 - 2.6]{BergelsonRichter2020}}]
	\label{lem_errorterm}
	Fix $\ve\in (0,1)$.
	\begin{enumerate}[(1)]
		\item Let $B\subset \P$ be a finite set of primes satisfying $\sum_{m\in B} 1/m \ge 1/\ve$. Then $
		\BEul{m\in B}\BEul{n\in B} \Phi(m,n)\,\le\, \ve.
		$
		
		\item  Let $P_1,P_2\subset \P$ be finite sets of primes satisfying $\sum_{p\in P_1} 1/p \ge {3}/{\ve}$ and $\sum_{q\in P_2} 1/q \ge {3}/{\ve}$. Let $B:= \{pq: p\in P_1,~ q\in P_2\}$. Then
		$
		\BEul{m\in B}\BEul{n\in B} \Phi(m,n)\,\le\, \ve.
		$
	\end{enumerate}
\end{lemma}

An integer $n\in \N$ is called a \textit{$k$-almost prime} if $\Omega(n)=k$. Let $\P_k:=\set{n\in \N:\Omega(n)=k}$ be the set of $k$-almost primes, $k\ge1$. Together with Lemma~\ref{lem_BR_keylemma2}, the following technical lemma guarantees the existence of two good finite sets $B_1$ and $B_2$ and will help us finish the proof of Theorem~\ref{thm_keythm_pmax}.

\begin{lemma}[{\cite[Lemma 2.2]{BergelsonRichter2020}}]
	\label{lem_BR_keylemma1}
	For all $\ve\in(0,1)$ and $\rho\in(1,1+\ve]$, there exist two finite and non-empty sets $B_1,B_2\subset\N$ satisfying the following properties: 
	\begin{enumerate}[(1)]
		\item\label{itm_a}
		$B_1\subset \P_1$ and $B_2\subset \P_2$;
		\item\label{itm_b}
		$|B_1\cap [\rho^j,\rho^{j+1})|=|B_2\cap [\rho^j,\rho^{j+1})|$ for all $j\in\N\cup\{0\}$;
		\item\label{itm_c}
		$\mathbb{E}^{\log}_{m\in B_1}\mathbb{E}^{\log}_{n\in B_1} \Phi(m,n)\le \ve$ and $\mathbb{E}^{\log}_{m\in B_2}\mathbb{E}^{\log}_{n\in B_2} \Phi(m,n)\le \ve$.
	\end{enumerate}
\end{lemma}

\begin{lemma}[{\cite[Lemma 2.3]{BergelsonRichter2020}}]
	\label{lem_BR_keylemma2}
	Fix $\ve\in(0,1)$ and $\rho\in(1,1+\ve]$. 
	Let $B_1$ and $B_2$ be two finite non-empty subsets of $\N$ satisfying property (2) in Lemma~\ref{lem_BR_keylemma1}. Then for any bounded arithmetic function $a\colon\N\to \C$ with $|a|\le 1$ we have 
	\begin{equation}
		\left|\,\BEul{p\in B_{1}} \, \BEu{n\in [\frac{N}{p}]} a(n) ~-~ \BEul{q\in B_{2}} \, \BEu{n\in [\frac{N}{q}]} a(n)\,\right| ~\le~ 3\ve.
	\end{equation}
\end{lemma}

\subsection{Proof of Theorem~\ref{thm_Wang2021}}
\label{sec_pf_Wang2021}

\begin{lemma}\label{lem_shift_invariant_pmax}
	Let $F,a:\N\to\C$ be two bounded functions.	Then for any integer $m\in\N$, we have that
	\begin{equation}\label{eqn_shift_invariant_pmax}
		\BEu{n\in[N]}F(p_{\max}(mn))a(n)=\BEu{n\in[N]}F(p_{\max}(n))a(n)+O(N^{-c}),
	\end{equation}
	for some constant $c>0$ depending only on $m$.
\end{lemma}

\begin{proof} Eq. \eqref{eqn_shift_invariant_pmax} is trivial for $m=1$. Suppose $m\ge2$.
	We break the average up into two parts as follows:
	\begin{align}
		\BEu{n\in[N]}F(p_{\max}(mn))a(n)&=\BEu{n\in[N]}1_{p_{\max}(n)\le m}F(p_{\max}(mn))a(n)+\BEu{n\in[N]}1_{p_{\max}(n)> m}F(p_{\max}(mn))a(n)\nonumber\\
		&=\BEu{n\in[N]}1_{p_{\max}(n)\le m}F(p_{\max}(mn))a(n)+\BEu{n\in[N]}1_{p_{\max}(n)> m}F(p_{\max}(n))a(n).\label{pmaxlem_pf_eq1}
	\end{align}
	For the first average, by the triangle inequality we have
	$$\abs{\BEu{n\in[N]}1_{p_{\max}(n)\le m}F(p_{\max}(mn))a(n)}\ll \BEu{n\in[N]}1_{p_{\max}(n)\le m}=\frac1N\Psi(N,m),$$
	where $\Psi(x,y):=\#\set{n\le x: p_{\max}(n)\le y}$ is the number of $y$-smooth numbers up to $x$.
	By Theorem~6 in \cite[Chapter 4]{TenenbaumMendes-France2000}, we have  $\Psi(x,y)\ll xe^{-u/2}$ with $u=\log x/\log y$ for $x\ge y\ge2$.  It follows  that $\frac1N\Psi(N,m)=O\of{e^{-c\log N}}=O(N^{-c})$ for some constant $c>0$ depending only on $m$, say $c=1/(2\log m)$. By \eqref{pmaxlem_pf_eq1}, we get that
	\begin{equation}\label{pmaxlem_pf_eq2}
		\BEu{n\in[N]}F(p_{\max}(mn))a(n)=\BEu{n\in[N]}1_{p_{\max}(n)> m}F(p_{\max}(n))a(n)+O(N^{-c}).
	\end{equation}
	Similarly, we have
	\begin{equation}\label{pmaxlem_pf_eq3}
		\BEu{n\in[N]}F(p_{\max}(n))a(n)=\BEu{n\in[N]}1_{p_{\max}(n)> m}F(p_{\max}(n))a(n)+O(N^{-c}).
	\end{equation}
	Hence \eqref{eqn_shift_invariant_pmax} follows immediately by \eqref{pmaxlem_pf_eq2} and \eqref{pmaxlem_pf_eq3}.
\end{proof}

\begin{theorem}\label{thm_keythm_pmax}
	Let $F:\R\to\C$ be a bounded function.	Then for any bounded arithmetic function $a:\N\to\C$, we have that
	\begin{equation}\label{eqn_keythm_pmax}
		\BEu{n\in[N]}F(p_{\max}(n))a(\Omega(n)+1)=\BEu{n\in[N]}F(p_{\max}(n))a(\Omega(n))+o_{N\to \infty}(1).
	\end{equation}
\end{theorem}

\begin{proof} Without loss of generality, we may assume that $|F|\le1$ and $|a|\le1$. Put 
	$$S_1:=\BEu{n\in[N]}F(p_{\max}(n))a(\Omega(n)+1) \quad\text{ and }\quad S_2:=\BEu{n\in[N]}F(p_{\max}(n))a(\Omega(n)).$$
	Let $\ve\in(0,1)$ and $\rho\in (1,1+\ve)$.  Let $B_1$ and $B_2$ be two finite non-empty sets satisfying the properties (1)-(3) in Lemma~\ref{lem_BR_keylemma1}. And we put
	\begin{align*}
		S_{B_1}&:=\BEul{p\in B_1}\, \BEu{n\in [\frac{N}{p}]}F(p_{\max}(pn))a(\Omega(pn)+1), \\
		\quad S_{B_2}&:=\BEul{q\in B_2}\, \BEu{n\in [\frac{N}{q}]}F(p_{\max}(qn))a(\Omega(qn)).
	\end{align*}
	Then by Proposition~\ref{prop_BR_inequality} and Lemma~\ref{lem_BR_keylemma1} (3), we get that
	\begin{equation}\label{pmaxpf_eq1}
		\limsup_{N\to\infty}|S_1-S_{B_1}|\le \ve^{1/2} \quad\text{and}\quad \limsup_{N\to\infty}|S_2-S_{B_2}|\le \ve^{1/2}.
	\end{equation}
	By Lemma~\ref{lem_BR_keylemma1} (1),  we have $\Omega(pn)=\Omega(n)+1$ and $\Omega(qn)=\Omega(n)+2$ for $p\in B_1, q\in B_2$. It follows that
	\begin{align*}
		S_{B_1}&=\BEul{p\in B_1}\, \BEu{n\in [\frac{N}{p}]}F(p_{\max}(pn))a(\Omega(n)+2), \\
		\quad S_{B_2}&=\BEul{q\in B_2}\, \BEu{n\in [\frac{N}{q}]}F(p_{\max}(qn))a(\Omega(n)+2).
	\end{align*}
	
	Since $B_1$ and $B_2$ are finite, by Lemma~\ref{lem_shift_invariant_pmax} there exists a constant $c>0$ depending only on $B_1\cup B_2$ such that
	\begin{equation}
		\BEu{n\in [\frac{N}{m}]}F(p_{\max}(mn))a(\Omega(n)+2)=\BEu{n\in[\frac{N}{m}]}F(p_{\max}(n))a(\Omega(n)+2)+O(N^{-c})
	\end{equation}
	for any $m\in B_1\cup B_2$. It follows that
	\begin{align}
		S_{B_1}&=\BEul{p\in B_1}\, \BEu{n\in [\frac{N}{p}]}F(p_{\max}(n))a(\Omega(n)+2)+O(N^{-c}), \\
		\quad S_{B_2}&=\BEul{q\in B_2}\, \BEu{n\in [\frac{N}{q}]}F(p_{\max}(n))a(\Omega(n)+2)+O(N^{-c}).
	\end{align}
	Then by Lemma~\ref{lem_BR_keylemma2} and Lemma~\ref{lem_BR_keylemma1} (2), we get that
	\begin{equation}\label{pmaxpf_eq2}
		\limsup_{N\to\infty} |S_{B_1}-S_{B_2}|\le3\ve.
	\end{equation}
	
	Combining \eqref{pmaxpf_eq1} and \eqref{pmaxpf_eq2}, we obtain
	\begin{equation}
		\limsup_{N\to\infty} |S_1-S_2|\le 2\ve^{1/2}+3\ve.
	\end{equation} 
	Since $\ve$ is arbitrary small,  we finally get that $\limsup_{N\to\infty} |S_1-S_2|=0$. This completes the proof.
\end{proof}

\begin{remark}
	For the special case $F=1$,	Eq. \eqref{eqn_keythm_pmax} is Richter's main theorem in \cite{Richter2021} where he gives a  new elementary proof of the Prime Number Theorem.
\end{remark}

\begin{proof}[Proof of Theorem~\ref{thm_Wang2021}]
	For any $x\in X$ and any bounded function $F:\N\to\C$ with $$\lim_{N\to\infty}\frac1N\sum_{n=1}^NF(p_{\max}(n))=\delta,$$ 
	we define the measure $\mu_N$ on $X$ by
	\[
	\mu_N:=\frac1N\sum_{n=1}^NF(p_{\max}(n))\delta_{T^{\Omega(n)}x}
	\]
	for $N \in \N$, where $\delta_y$ denotes the point mass at $y$ for any $y\in X$. Define $\mu'=\delta\cdot \mu$. Then Eq. \eqref{eqn_Wang2021} is equivalent to the convergence of the sequence  $\{\mu_N \}_{n \in \N}$ to $\mu'$ in the weak-$\ast$ topology. Since $\mu$ is uniquely ergodic, to prove $\mu_N\to\mu'$ it suffices to show that each limit point of $\{\mu_N\}_{n \in \N}$ is $T$-invariant. This is done by the following $T$-invariance
	\begin{equation}\label{eqn_T-invariance}
		\lim_{N \to \infty} \Big| \int_X f \circ T\, d\mu_N - \int_X f  \, d\mu_N  \Big| = 0
	\end{equation}
	for all $f\in C(X)$. On the other hand, if we take $a(n)=f(T^{n}x)$ in Theorem~\ref{thm_keythm_pmax}, then we get the following equivalent form of Eq.~\eqref{eqn_T-invariance}:
	\begin{equation}\label{key}
		\lim_{N \to \infty} \Big| \frac1N\sum_{n=1}^NF(p_{\max}(n))f(T^{\Omega(n)+1}x)- \lim_{N\to\infty}\frac1N\sum_{n=1}^NF(p_{\max}(n))f(T^{\Omega(n)}x) \Big| = 0.
	\end{equation}
	This completes the proof of Theorem~\ref{thm_Wang2021}.
\end{proof}

\section{Bergelson and Richter's techniques for additive arithmetic semigroups}
\label{sec_BR_techniques_ff}

In this section, we  show the analogue of Bergelson and Richter's techniques for additive arithmetic groups. Let $\cG$ be an additive arithmetic semigroup Axiom $\mathcal{A}^\#$ type satisfying \eqref{eqn_axiom} and \eqref{eqn_classicalPNT}:
\begin{align}
	G(n)&=A q^n+O(q^{\eta n}),	\label{eqn_axiom_ff}\\
	\bar{\Lambda}(n)&=q^n+O(q^{\theta n})\label{eqn_classicalPNT_ff}
\end{align}
for some constants $A>0, q>1,0\le\eta<\theta<1$. Let $a: \cG\to \C$ be a function on $\cG$. For any  finite and non-empty subset $B\subset \cG$, we define the \textit{Ces\`aro average} of $a$ over $B$ and the \textit{logarithmic average} of $a$ over $B$ respectively by
$$\BEu{g\in B}a(g):=\frac1{|B|}\sum_{g\in B}a(g) \qquad\text{and }\qquad \BEul{g\in B}a(g):=\frac{\sum_{g\in B}{a(g)}/{|g|}}{\sum_{g\in B}{1}/{|g|}}.$$

First, we show the analogue of Proposition~\ref{prop_BR_inequality}.

\begin{proposition}\label{prop_BR_inequality_ff}
	Let	$B\subset \cG$ be a finite and non-empty subset of $\cG$. Then for any bounded function $a: \cG\to \C$ with $|a|\le1$, we have
	\begin{equation}\label{eqn_BR_inequality_ff}
		\limsup_{n\to\infty}\abs{\BEu{g\in\cG_n}a(g)-\BEul{h\in B}\BEu{ g\in\cG_{n-\partial(h)}}a(gh)}\le \Big(\BEul{g\in B}\BEul{h\in B}\Phi(g,h)\Big)^{1/2}
	\end{equation}
	holds for any $n$, where $\Phi(g,h)=|\gcd(g,h)|-1$. 
\end{proposition}

\begin{proof} First, by Eq. \eqref{eqn_axiom_ff}, we observe that  the estimate $|\cG_n|\sim A'q^n \,\,(n\to\infty)$ holds with $A'=Aq/(q-1)$.   Then for any $h\in B$, by $|h|=q^{\pl(h)}$ we have that
	$$\BEu{g\in\cG_{n-\partial (h)}}a(gh)\sim \frac{|h|}{A'q^{n}} \sum_{g\in\cG_{n-\partial (h)}}a(gh)\sim  \BEu{g\in\cG_n} a(g)|h|1_{h|g} \,\,(n\to\infty).$$
	
	Since $B$ is finite, it follows that
	$$\limsup_{n\to\infty}\abs{\BEu{g\in\cG_n}a(g)-\BEul{h\in B}\BEu{ g\in\cG_{n-\partial(h)}}a(gh)}=\limsup_{n\to\infty}\abs{\BEu{g\in\cG_n}\BEul{h\in B} a(g)(1-|h|1_{h|g})}.$$
	Then by the triangle inequality and the Cauchy-Schwarz inequality,
	\begin{align}
		&\limsup_{n\to\infty}\abs{\BEu{g\in\cG_n}\BEul{h\in B} a(g)(1-|h|1_{h|g})}\nonumber\\
		\le&\limsup_{n\to\infty}\BEu{g\in\cG_n}\abs{\BEul{h\in B} (1-|h|1_{h|g})}\nonumber\\
		\le&\limsup_{n\to\infty}\Big(\BEu{g\in\cG_n}\Big|\BEul{h\in B} (1-|h|1_{h|g})\Big|^2\Big)^{1/2}.\label{keyprop_pf1}
	\end{align}
	
	Now, we expand out the square on the right-hand side of \eqref{keyprop_pf1} and write
	$$
	\BEu{g\in\cG_n}\Big|\BEul{h\in B} (1-|h|1_{h|g})\Big|^2:=1-2\Sigma_1+\Sigma_2,
	$$
	where $\Sigma_1=\BEu{g\in\cG_n}\BEul{h\in B}|h|1_{h|g}$ and $\Sigma_2=\BEu{g\in\cG_n}\BEul{h\in B}\BEul{h'\in B}|h|1_{h|g}|h'|1_{h'|g}$.
	
	For $\Sigma_1$, we have $\Sigma_1=\BEul{h\in B}\BEu{g\in\cG_n}|h|1_{h|g}$ and $$\BEu{g\in\cG_n}|h|1_{h|g}=\frac{|h|}{|\cG_n|}\sum_{g\in\cG_n}1_{h|g}=\frac{|h||\cG_{n-\pl(h)}|}{|\cG_n|}\sim 1 \,\,(n\to\infty).$$
	It follows that $\Sigma_1\sim1 \,\,(n\to\infty)$. For $\Sigma_2$, we have $$\BEu{g\in\cG_n}|h|1_{h|g}|h'|1_{h'|g}=\frac{|h||h'||\cG_{n-\pl(\lcm(h,h'))}|}{|\cG_n|}\sim\frac{|h||h'|}{|\lcm(h,h')|}=|\gcd(h,h')| \,\,(n\to\infty).$$ 
	And then $\Sigma_2\sim \BEul{h\in B}\BEul{h'\in B}|\gcd(h,h')| \,\,(n\to\infty)$. Therefore,
	\begin{equation}\label{keyprop_pf2}
		\lim_{n\to\infty}\BEu{g\in\cG_n}\Big|\BEul{h\in B} (1-|h|1_{h|g})\Big|^2=\BEul{g\in B}\BEul{h\in B}\Phi(g,h).
	\end{equation}
	Then \eqref{eqn_BR_inequality_ff} follows immediately by combining \eqref{keyprop_pf1} and \eqref{keyprop_pf2}.
\end{proof}

Set $E(B):=\BEul{g\in B}\BEul{h\in B}\Phi(g,h)$ for the error term in \eqref{eqn_BR_inequality_ff}. If two random chosen elements in $B$ have a ``high chance" of being coprime to each other, then we expect $E(B)$ to be small. For $g\in\cG$, we call $g$ a \textit{$k$-almost prime} if $\Omega(g)=k$. Let $\cP_k=\set{g\in\cG:\Omega(g)=k}$ be the set of $k$-almost primes in $\cG$, $k\ge1$. Then the segments of $\cP_k$ are our desired sets. Similar to Lemma~\ref{lem_errorterm}, to estimate the size of $E(B)$, we have the following criterion.

\begin{lemma}\label{lem_errorterm_ff}
	Fix $0<\ve<1$. 
	\begin{enumerate}[(1)]
		\item If $B\subset \cP$ and $\sum_{g\in B}1/|g|\ge1/\ve$, then $E(B)\le\ve$.
		
		\item If $B\subset \cP_2$ and $B=P_1\cdot P_2$ for $P_i\subset\cP$ with $\sum_{g\in P_i}1/|g|\ge3/\ve$, $i=1,2$, then $E(B)\le\ve$. 
	\end{enumerate}	
\end{lemma}

\begin{proof}
	See \cite[Lemma 2.5-2.6]{BergelsonRichter2020}.
\end{proof}

Now, we prove the analogue of Lemmas~\ref{lem_BR_keylemma1} and \ref{lem_BR_keylemma2}. Like Bergelson and Richter use the $\rho$-adic intervals $[\rho^j,\rho^{j+1})\ (j\ge0)$ to break the line $[1,\infty)$, we use the set $\cG(j)$ of elements of degree $j$ for $j\ge0$ to break the semigroup $\cG$. In each stage $j$, we have sufficiently many primes for sufficiently large $j$.

\begin{lemma} \label{lem_BR_keylemma1_ff}
	For all $0<\ve<1$, there exist two finite and non-empty sets $B_1,B_2\subset\cG$ with the following properties:
	\begin{enumerate}[(1)]
		\item $B_1\subset\cP_1$ and $B_2\subset\cP_2$;
		
		\item $|B_1\cap\cG(n)|=|B_2\cap\cG(n)|$ for all $n\ge0$;
		
		\item $E(B_1)\le\ve$ and $E(B_2)\le\ve$.
	\end{enumerate}	
\end{lemma}

\begin{proof}
	Take $j_0\ge1$ such that $q^{j_0}\ge2$, then ${q^n}/{n}$ is increasing for $n\ge j_0$. By Eq. \eqref{eqn_classicalPNT_ff}, which is the classical PNT, there exists some constant $C>0$ such that $\pi_\cG(n)\ge Cq^n/n$ for all $n\ge1$.	
	
	First, take $s$ large enough such that $\sum_{j_0\le \ell<s}C/\ell\ge3/\ve$. 
	Let $P_{1,\ell}=\cP(\ell)$ for $j_0\le\ell<s$ and let $P_1=\bigcup_{j_0\le\ell<s}P_{1,\ell}$.
	
	Second, take $t$ large enough such that $\sum_{1\le j\le t}\frac{C}{2|P_1|sj}\ge3/\ve$.  For each $1\le j\le t$, choose a subset $P_{2,j}\subset\cP(sj)$  of size 
	$$\frac{Cq^{sj}}{2|P_1|sj}\le |P_{2,j}|\le\frac{Cq^{sj}}{|P_1|sj}.$$
	
	Then $P_{1,\ell}\cdot P_{2,j}\subset\cG(sj+\ell)$ of size $|P_{1,\ell}\cdot P_{2,j}|$. Since $|P_{2,j}|\le\frac{Cq^{sj}}{|P_1|sj}$, we get that $$|P_{1,\ell}\cdot P_{2,j}|\le \frac{Cq^{sj}}{sj} \le \frac{Cq^{sj+\ell}}{sj+\ell}\le P(sj+\ell)$$ for $j_0\le \ell <s$ and $1\le j \le t$. It follows that exists some subset $Q_{\ell,j}\subset\cP(sj+\ell)$ of size $|P_{1,\ell}\cdot P_{2,j}|$. 
	
	Now, we take $$B_1=\bigcup_{\substack{j_0\le\ell<s\\1\le j\le t}}Q_{\ell,j} \text{ and } B_2=\bigcup_{\substack{j_0\le\ell<s\\1\le j\le t}}P_{1,\ell}\cdot P_{2,j}.$$ It is readily to see that properties (1) and (2) hold.  Let $P_1=\bigcup_{j_0\le\ell<s}P_{1,\ell}$ as before and $P_2:=\bigcup_{1\le j\le t}P_{2,j}$, then $B_2=P_1\cdot P_2$. 
	
	Finally, we estimate $E(B_1)$ and $E(B_2)$. By the choice of $s$ and $t$, $$\sum_{p\in P_1}\frac1{|p|}=\sum_{j_0\le\ell<s}\frac{\pi_\cG(\ell)}{q^\ell}\ge\sum_{j_0\le\ell<s}\frac{C}{\ell}\ge\frac3\ve;$$ 
	and $$\sum_{p\in P_2}\frac1{|p|}=\sum_{1\le j\le t}\frac{|P_{2,j}|}{q^{sj}}\ge\sum_{1\le j\le t}\frac{C}{2|P_1|sj}\ge\frac3\ve.$$ 
	By Lemma~\ref{lem_errorterm_ff}(2), we get that $E(B_2)\le \ve$. As regards $B_1$, since $|B_1\cap\cP(n)|=|B_2\cap\cP(n)|$ for all $n\ge0$ by property (2),  we get that
	$$\sum_{g\in B_1}\frac1{|g|}=\sum_{g\in B_2}\frac1{|g|}=\Big(\sum_{p\in P_1}\frac1{|p|}\Big)\Big(\sum_{q\in P_2}\frac1{|q|}\Big)\ge\big(\frac3{\ve}\big)^2\ge\frac1\ve.$$ By Lemma~\ref{lem_errorterm_ff}(1), we have $E(B_1)\le \ve$. 	Thus, property (3) also holds.
\end{proof}

\begin{lemma}\label{lem_BR_keylemma2_ff}
	Let $B_1$ and $B_2$ be two finite and nonempty subsets of $\cG$ with the property that $|B_1\cap\cG(j)|=|B_2\cap\cG(j)|$ for all $j\ge0$. Then for any function $a:\cG\to\C$ we have 
	\begin{equation}\label{keylemma2_eq}
		\BEul{h\in B_1}\BEu{ g\in\cG_{n-\partial(h)}}a(g)=\BEul{h\in B_2}\BEu{ g\in\cG_{n-\partial(h)}}a(g).
	\end{equation}
\end{lemma}
\begin{proof} We split up the logarithmic averages over $B_1$ and $B_2$ according to the degrees of their elements, and then interchange the order of summations
	\begin{align*}
		&\BEul{h\in B_1}\BEu{ g\in\cG_{n-\partial(h)}}a(g)-\BEul{h\in B_2}\BEu{ g\in\cG_{n-\partial(h)}}a(g)\\
		=&\sum_{j=0}^\infty \Big(\BEul{h\in B_1}1_{\pl(h)=j}\BEu{ g\in\cG_{n-j}}a(g)-\BEul{h\in B_2}1_{\pl(h)=j}\BEu{ g\in\cG_{n-j}}a(g)\Big)\\
		=&\sum_{j=0}^\infty \BEu{ g\in\cG_{n-j}}a(g) \Big(\BEul{h\in B_1}1_{\pl(h)=j}-\BEul{h\in B_2}1_{\pl(h)=j}\Big).
	\end{align*}	
	Since $|B_1\cap\cG(j)|=|B_2\cap\cG(j)|$, we get that $\BEul{h\in B_1}1_{\pl(h)=j}=\BEul{h\in B_2}1_{\pl(h)=j}$ for all $j\ge0$. Thus, \eqref{keylemma2_eq} holds.
\end{proof}

\section{Proof of Theorems~\ref{thm_Loyd2021_ff} and \ref{thm_Wang2021_ff}}
\label{sec_proof_mainthm_ff}

In this section, we prove Theorems~\ref{thm_Loyd2021_ff} and \ref{thm_Wang2021_ff}. The proofs of these two theorems are similar and parallel. Taking $F=1$ in the following context, one gets the proof of Eq.~\eqref{eqn_BR2020_ff}. Hence it suffices to prove \eqref{eqn_Loyd2021_ff} and \eqref{eqn_Wang2021_pmax_ff}. Due to the  unique ergodicity of $(X,\mu,T)$,  similar to the argument in the proof of Theorem~\ref{thm_Wang2021}, to prove \eqref{eqn_Loyd2021_ff} and \eqref{eqn_Wang2021_pmax_ff}, it suffices to show the following two equations of $T$-invariance
\begin{align}
	&\lim_{n\to\infty}\abs{\BEu{g\in\cG_n}F(\varphi(g))f(T^{\Omega(g)+1}x)-\BEu{g\in\cG_n}F(\varphi(g))f(T^{\Omega(g)}x)}=0;\label{eqn_inv_Loyd}\\
	&\lim_{n\to\infty}\abs{\BEu{g\in\cG_n^+}F(p_{\max}(g))f(T^{\Omega(g)+1}x)-\BEu{g\in\cG_n^+}F(p_{\max}(g))f(T^{\Omega(g)}x)}=0\label{eqn_inv_Wang}
\end{align}
for every $f\in C(X)$ and $x\in X$, where we recall that $\varphi(g):=\frac{\Omega(g)-\log n}{\sqrt{\log n}}$. Here we clarify that $F\in C_c(\R)$ in \eqref{eqn_inv_Loyd} but we assume it to be a bounded function on $\cG$ in \eqref{eqn_inv_Wang}. To prove \eqref{eqn_inv_Loyd} and \eqref{eqn_inv_Wang}, it suffices to prove the analogues of Theorem~\ref{thm_keythm_pmax} for $F(\varphi(g))$ and $F(p_{\max}(g))$ respectively, which will be done in Theorems~\ref{thm_Loyd2021_keythm_ff} and \ref{thm_Wang2021_keythm_ff} below. First, using the ideas in \cite[Proposition 4.5]{Loyd2021}, we  prove a lemma that will be used in the proof of Theorem~\ref{thm_Loyd2021_keythm_ff}.
\begin{lemma}\label{lem_shift_invariant_phi}
	Let $F\in C_c(\R)$ be a compactly supported continuous function.	Let $a:\cG\to\C$ be a bounded function on $\cG$.
	Then for any fixed $h\in \cG$, we have 
	\begin{equation}\label{eqn_shift_invariant_phi}
		\BEu{ g\in\cG_n}F(\varphi(hg))a(g)=\BEu{ g\in\cG_n}F(\varphi(g))a(g)+o_{n\to\infty}(1).
	\end{equation}
\end{lemma}
\begin{proof}
	We may assume that $|a|\le1$. By definition, we have
	$\varphi(hg)-\varphi(g)=\Omega(h)/\sqrt{\log n}$ for all $g\in \cG$. Since every compactly supported continuous function is uniformly continuous on $\R$ and $F\in C_c(\R)$, for any $\ve>0$, there exists some $N=N(\ve)>0$ such that $|F(\varphi(hg))-F(\varphi(g))|<\ve$ holds uniformly for all $n\ge N$ and all $g\in \cG$. In particular, for all $n\ge N$ and all $g\in \cG_n$, we have $|F(\varphi(hg))-F(\varphi(g))|<\ve$. Then for $n\ge N$, we have
	\[
	\left|\BEu{ g\in\cG_n}F(\varphi(hg))a(g)-\BEu{ g\in\cG_n}F(\varphi(g))a(g)\right|\le \BEu{ g\in\cG_n}|F(\varphi(hg))-F(\varphi(g))|\le  \BEu{ g\in\cG_n}\ve=\ve.
	\]
	Hence \eqref{eqn_shift_invariant_phi} follows.
\end{proof}

Now, we prove the following analogue of Theorem~\ref{thm_keythm_pmax} for $F(\varphi(g))$.

\begin{theorem}\label{thm_Loyd2021_keythm_ff}
	Let $F\in C_c(\R)$ be a compactly supported continuous function.	Then for any bounded arithmetic function $a:\N\to\C$, we have that
	\begin{equation}\label{eqn_Loyd2021_keythm_ff}
		\BEu{g\in\cG_n}F(\varphi(g))a(\Omega(g)+1)=\BEu{g\in\cG_n}F(\varphi(g))a(\Omega(g))+o_{n\to \infty}(1).
	\end{equation}
\end{theorem}

\begin{proof} Without loss of generality, we may assume that $|F|\le1$ and $|a|\le1$. For any $0<\ve<1$, we take $B_1$ and $B_2$ as in Lemma~\ref{lem_BR_keylemma1_ff}. Then by Proposition~\ref{prop_BR_inequality_ff}, 
	\begin{align}
		&\limsup_{n\to\infty}\abs{\BEu{g\in\cG_n}F(\varphi(g))a(\Omega(g)+1)-\BEul{p\in B_1}\BEu{ g\in\cG_{n-\partial(p)}}F(\varphi(pg))a(\Omega(pg)+1)}\le\ve^{1/2}; \label{eqn1_pf_Loyd2021_keythm_ff}\\
		&\limsup_{n\to\infty}\abs{\BEu{g\in\cG_n}F(\varphi(g))a(\Omega(g))-\BEul{q\in B_2}\BEu{ g\in\cG_{n-\partial(q)}}F(\varphi(qg))a(\Omega(qg))}\le\ve^{1/2}. \label{eqn2_pf_Loyd2021_keythm_ff}
	\end{align}
	
	By Lemma~\ref{lem_BR_keylemma1_ff} (1), we have $a(\Omega(pg)+1)=a(\Omega(g)+2)$ and $a(\Omega(qg))=a(\Omega(g)+2)$ for any $p\in B_1$ and $q\in B_2$.  Then for any $h\in B_1\cup B_2$, by Lemma~\ref{lem_shift_invariant_phi}, we get that
	\begin{equation}\label{eqn3_pf_Loyd2021_keythm_ff}
		\BEu{ g\in\cG_{n-\partial(h)}}F(\varphi(hg))a(\Omega(g)+2)=\BEu{ g\in\cG_{n-\partial(h)}}F(\varphi(g))a(\Omega(g)+2)+o_{n\to\infty}(1).
	\end{equation}
	
	Due to the finiteness of $B_1$ and $B_2$, from \eqref{eqn1_pf_Loyd2021_keythm_ff}-\eqref{eqn3_pf_Loyd2021_keythm_ff} we obtain that
	\begin{align}
		&\limsup_{n\to\infty}\abs{\BEu{g\in\cG_n}F(\varphi(g))a(\Omega(g)+1)-\BEul{h\in B_1}\BEu{ g\in\cG_{n-\partial(h)}}F(\varphi(g))a(\Omega(g)+2)}\le\ve^{1/2};\label{pf_eq1}\\
		&\limsup_{n\to\infty}\abs{\BEu{g\in\cG_n}F(\varphi(g))a(\Omega(g))-\BEul{h\in B_2}\BEu{ g\in\cG_{n-\partial(h)}}F(\varphi(g))a(\Omega(g)+2)}\le\ve^{1/2}.\label{pf_eq2}
	\end{align}
	
	By Lemma~\ref{lem_BR_keylemma2_ff}, we have 
	$$\BEul{h\in B_1}\BEu{ g\in\cG_{n-\partial(h)}}F(\varphi(g))a(\Omega(g)+2)=\BEul{h\in B_2}\BEu{ g\in\cG_{n-\partial(h)}}F(\varphi(g))a(\Omega(g)+2).$$ 
	Therefore,  by \eqref{pf_eq1}, \eqref{pf_eq2} and the triangle inequality,
	$$\limsup_{n\to\infty}\abs{\BEu{g\in\cG_n}F(\varphi(g))a(\Omega(g)+1)-\BEu{g\in\cG_n}F(\varphi(g))a(\Omega(g))}\le2\ve^{1/2}.$$
	Since $\ve$ is arbitrarily small, we obtain \eqref{eqn_Loyd2021_keythm_ff} as desired.	
\end{proof}

Now, to show the analogue of Theorem~\ref{thm_keythm_pmax} for $F(p_{\max}(g))$ over $g\in \cG^+$, we extend the domain of $p_{\max}$ to the whole $\cG$ by defining $p_{\max}(g):=1$ for all $g\in \cG\setminus\cG^+$. Under this definition, it has the property that $p_{\max}(gh)=p_{\max}(g)$ for any $g,h\in \cG$ with $\pl_{\max}(g)>\pl_{\max}(h)$. Then we have the following analogue of Lemma~\ref{lem_shift_invariant_pmax} for $F(p_{\max}(g))$ over $g\in\cG$.

\begin{lemma}\label{lem_shift_invariant_plmax}
	Let $F, a:\cG\to\C$ be two bounded arithmetic functions  on $\cG$.
	Then for any $h\in \cG$, we have 
	\begin{equation}\label{eqn_shift_invariant_plmax}
		\BEu{ g\in\cG_n}F(p_{\max}(hg))a(g)=\BEu{ g\in\cG_n}F(p_{\max}(g))a(g)+O(e^{-cn})
	\end{equation}
	for some constant $c>0$.
\end{lemma}
\begin{proof}
	We may assume that $|F|\le1$ and $|a|\le1$. Let $k:=\max\set{2/\log q, \pl_{\max}(h)}$. Then we break the left average up into two parts as follows:
	\begin{align}
		\BEu{ g\in\cG_n}F(p_{\max}(hg))a(g)&=\BEu{ g\in\cG_n}1_{\pl_{\max}(g)\le k}F(p_{\max}(hg))a(g)+\BEu{ g\in\cG_n}1_{\pl_{\max}(g)>k}F(p_{\max}(hg))a(g)\nonumber\\
		&=\BEu{ g\in\cG_n}1_{\pl_{\max}(g)\le k}F(p_{\max}(hg))a(g)+\BEu{ g\in\cG_n}1_{\pl_{\max}(g)>k}F(p_{\max}(g))a(g).
	\end{align}
	By the triangle inequality and Lemma~\ref{lem_smooth_numbers}, we have
	\begin{equation}
		\left|\BEu{ g\in\cG_n}1_{\pl_{\max}(g)\le k}F(p_{\max}(hg))a(g)\right|\le \BEu{ g\in\cG_n}1_{\pl_{\max}(g)\le k}\ll \frac1{q^n}\Psi_{\cG}(n,k)\ll e^{-\frac{n}{2k}}.
	\end{equation}
	
	It follows that 
	\begin{equation}\label{eqn1_pf_plmax}
		\BEu{ g\in\cG_n}F(p_{\max}(hg))a(g)=\BEu{ g\in\cG_n}1_{\pl_{\max}(g)>k}F(p_{\max}(g))a(g)+O(e^{-cn})
	\end{equation}
	for some constant $c>0$, say $c=1/(2k)$. Similarly,
	\begin{equation}\label{eqn2_pf_plmax}
		\BEu{ g\in\cG_n}F(p_{\max}(g))a(g)=\BEu{ g\in\cG_n}1_{\pl_{\max}(g)>k}F(p_{\max}(g))a(g)+O(e^{-cn}).
	\end{equation}
	Hence \eqref{eqn_shift_invariant_plmax} follows immediately by \eqref{eqn1_pf_plmax} and \eqref{eqn2_pf_plmax}.
\end{proof}

By Lemma~\ref{lem_shift_invariant_plmax}, we see that Eq. \eqref{eqn_shift_invariant_phi} holds for $F\circ\varphi$ replaced by $F\circ p_{\max}$. Therefore, following the proof of Theorem~\ref{thm_Loyd2021_keythm_ff}, we still have
\begin{equation}\label{eqn_Wang2021_keythm_G}
	\BEu{g\in\cG_n}F(p_{\max}(g))a(\Omega(g)+1)=\BEu{g\in \cG_n}F(p_{\max}(g))a(\Omega(g))+o_{n\to \infty}(1).
\end{equation}
By Lemma~\ref{lem_size_of_cG}, the following identity 
\begin{equation}\label{eqn_cG_cGplus}
	\BEu{g\in\cG_n}w(g)=\BEu{g\in\cG_n^+}w(g)+o_{n\to \infty}(1)
\end{equation}
holds for any bounded function $w:\cG\to\C$.
Thus, from \eqref{eqn_Wang2021_keythm_G} and \eqref{eqn_cG_cGplus} we get the following analogue of Theorem~\ref{thm_keythm_pmax} for $\cG^+$, which implies Eq.~\eqref{eqn_inv_Wang} and hence Theorem~\ref{thm_Wang2021_ff}.

\begin{theorem}\label{thm_Wang2021_keythm_ff}
	Let $F:\cG\to\C$ be a bounded  function.
	Then for any bounded arithmetic function $a:\N\to\C$, we have that
	\begin{equation}\label{eqn_Wang2021_keythm_ff}
		\BEu{g\in\cG_n^+}F(p_{\max}(g))a(\Omega(g)+1)=\BEu{g\in \cG_n^+}F(p_{\max}(g))a(\Omega(g))+o_{n\to \infty}(1).
	\end{equation}
\end{theorem}

\begin{remark}
	Let $(X,\mu,T)$ be uniquely ergodic.	Similar to the proof of Theorem~\ref{thm_Wang2021_ff} above, we have that  if $F:\N\to\C$ is a bounded arithmetic function such that 
	$$\sum_{g\in \cG_n}F(\pl_{\max}(g))\sim \delta \cdot |\cG_n|,$$
	then we have
	\begin{equation}
		\lim_{n\to\infty}\frac1{|\cG_n|}\sum_{g\in \cG_n}F(\pl_{\max}(g))f(T^{\Omega(g)}x)=\delta\int_X f\,d\mu
	\end{equation}
	for every $x\in X$ and $f\in C(X)$.
\end{remark}

\section*{Acknowledgment}
The author would like to thank Fei Wei and Shaoyun Yi for their comments and encouragement. The author is also grateful to the anonymous referees for
their helpful corrections and suggestions.

\end{document}